\documentclass{amsart}
\usepackage{latexsym}
\usepackage[nohug]{diagrams}
\usepackage{amsfonts}
\usepackage{amsmath}
\usepackage{amssymb}
\usepackage{amsthm}
\usepackage[mathscr]{eucal}
\newtheorem{theorem}{Theorem}[section]
\newtheorem{lemma}[theorem]{Lemma}
\newtheorem{proposition}[theorem]{Proposition}

\theoremstyle{definition}

\newtheorem*{example*}{Example}

\theoremstyle{remark}
\newtheorem{remark}[theorem]{Remark}
\newtheorem*{remark*}{Remark}
\numberwithin{equation}{section}
\newcommand{\la}{\label}

\DeclareMathOperator{\Out}{\sf Out}
\DeclareMathOperator{\Pic}{\sf Pic}

\DeclareMathOperator{\End}{\sf End}
\DeclareMathOperator{\Aut}{\sf Aut}

\DeclareMathOperator{\Frac}{\sf Frac}
\DeclareMathOperator{\gr}{gr}

\DeclareMathOperator{\rk}{rk}
\DeclareMathOperator{\Hilb}{\sf Hilb}
\DeclareMathOperator{\id}{Id}
\DeclareMathOperator{\len}{length}
\DeclareMathOperator{\ol}{1-length}

\DeclareMathOperator{\Der}{Der}
\DeclareMathOperator{\Div}{Div}
\DeclareMathOperator{\Hom}{Hom}

\def\O{\mathcal{O}}
\def\R{\mathcal{R}}
\def\C{\mathcal{C}}
\def\K{\mathbb{K}}
\def\F{\mathcal{F}}
\def\C{\mathcal{C}}

\def\A{\mathbb{A}}
\def\N{\mathbb{N}}
\def\Z{\mathbb{Z}}

\def\D{\mathcal{D}}

\def\c{\mathbb{C}}
\def\I{\mathscr{I}}

\def\DO{\overline{\mathcal{D}}}
\def\MO{\overline{M}}
\def\AO{\overline{A}}
\def\RO{\overline{\mathcal{R}}}
\def\gammaO{\overline{\gamma}}
\def\m{\mathfrak{m}}
\begin{document}
\title[Ideal classes of differential operators]
{Differential operators on an affine curve: ideal classes and Picard groups}

    \author{Yuri Berest}
%
%    address of record for the research reported here
%
    \address{Department of Mathematics, Cornell University,
        Ithaca, NY 14853, USA}
    \email{berest@math.cornell.edu}
     \author{George Wilson}
     \address{Mathematical Institute, 24--29 St Giles, Oxford OX1 3LB, UK}
      \email{wilsong@maths.ox.ac.uk}
%
%\date{24 September 2008}
\begin{abstract}
Let $\, X \,$ be a smooth complex affine curve, and let $\, \R \,$ be
the space of right ideal classes in the ring $\, \D \,$ of
differential operators on $\, X \,$.  We introduce and study a fibration
$\, \gamma : \R \to \Pic X \,$.  We relate this fibration to the
corresponding one in the classical limit, and derive an integer invariant
$\, n \,$ which indexes the decomposition of the fibres of $\, \gamma \,$
into Calogero-Moser spaces (see \cite{BC}).
We also study the action of the group
$\, \Pic \D \,$ on our fibration; and we explain how to define $\, \gamma \,$
in the framework of the Grassmannian description of $\, \R \,$ due to
Cannings and Holland.
\end{abstract}
\maketitle
\section{Introduction}

Let $\, X \,$ be an irreducible smooth affine algebraic curve over $\, \c \,$,
and let $\, \D \,$ be the ring of differential operators
on $\, \O(X) \,$.  We denote by $\, \R \,$ the set of isomorphism classes
of right ideals in $\, \D \,$ (equivalently: of finitely generated
torsion-free right $\, \D$-modules of rank $1$).  In the case
$\, X = \A^1 $, where $\, \D \,$ is the Weyl algebra,
$\, \R \,$ can be described in two essentially different ways
(i) as the {\it adelic Grassmannian} of \cite{W1} (see \cite{CH1});
(ii) as the disjoint union of certain finite-dimensional algebraic
varieties, the {\it Calogero-Moser spaces} $\, \C_n $ (see \cite{BW1},
\cite{BW2}).
The general case is similar (see \cite{BC}, \cite{Be}, \cite{W3}, \cite{BN});
however, some new features come into play.
In this preliminary paper we make a fairly thorough
study of one of them, which comes from the fact
that if $\, X \,$ has genus $>0$ then the Picard group $\, \Pic X \,$
is nontrivial.  We recall that $\, \Pic X \,$ is the group of algebraic
line bundles (or divisor classes) over $\, X \,$; alternatively,
we may think of it as the group of ideal classes in the Dedekind domain
$\, \O(X) \,$.  We shall see that there is a natural fibration
\begin{equation*}
\label{gam}
\gamma \,:\, \R \to \Pic X \ .
\end{equation*}
The significance of this fibration is that it is the individual fibres
of $\, \gamma \,$ (rather than the whole of $\, \R \,$) that have
descriptions generalizing those already known in the case $\, X = \A^1 \,$.
Of course, in that case $\, \Pic X \,$ is trivial, so there is only
one fibre.

The definition of $\, \gamma \,$ is via $K$-theory: if we wish
to classify the ideals in $\, \D \,$ up to isomorphism, we should
certainly understand first how to classify them up to {\it stable}
isomorphism.  Let $\, K_0(\D) \,$ be the Grothendieck group of finitely
generated projective right $\, \D$-modules: since $\, \D \,$ is
Noetherian and hereditary, every ideal $\, M \subseteq \D \,$ is finitely
generated and projective, hence determines a class $\, [M] \in  K_0(\D) \,$.
The inclusion of $\, \O(X) \,$ in $\, \D \,$ as the
operators of order zero induces a map from $\, K_0(X) \,$
($\, \equiv K_0(\O(X)) \,$)
to $\, K_0(\D) \,$; a theorem of Quillen
(see \cite{Q}, p.~120, Theorem~7) shows that this map
is an {\it isomorphism}.  On the other hand, it is well known that the map
$$
\rk \oplus \det \,:\, K_0(X) \to \Z \oplus \Pic X
$$
is also an isomorphism.  Since $\, \rk M = 1 \,$, it follows that
the stable isomorphism
class of $\, M \,$ is determined by its component in $\, \Pic X \,$:
we define $\, \gamma(M) \,$ to be this component.  That is, if $\, M \,$
is a right ideal in $\, \D \,$, then $\, \gamma(M) \, ${\it is the
unique ideal class $\, (I) \in \Pic X \,$ such that $\, M \,$ is
stably isomorphic to} $\, I\D \,$.

Now, as for any algebra $\, \D \,$,
the group $\, \Pic \D \,$ of invertible $\, \D$-bimodules acts
(compatibly) on  modules and on $\, K_0(\D) \,$  by tensor product.
Thus we have natural actions of $\, \Pic \D \,$ on $\, \R \,$ and on
$\, \Pic X \,$, and the map $\, \gamma \,$ is equivariant with respect to
these actions.  The following fact
perhaps justifies referring to $\, \gamma \,$ as a ``fibration''.
\begin{theorem}
\label{T1}
The action of $\, \Pic \D \,$ on $\, \Pic X \,$ is transitive.  Thus
for any two fibres $\, F_1 \,$ and $\, F_2 \,$ of $\, \gamma \,$ there
are elements of the group $\, \Pic \D \,$ that map $\, F_1 \,$
bijectively onto $\, F_2 \,$.
\end{theorem}

Next, we wish to compare our situation with the corresponding one in
the ``classical limit''.  Recall that if we give $\, \D \,$ its usual
filtration by order of operators,
then the associated graded algebra $ \, \DO := \gr \D \,$ is canonically
identified with the commutative algebra
$\, \O(T^*X) \,$ of functions on the cotangent bundle of $\, X \,$.
Thus our space $\, \R \,$ is a
noncommutative analogue of the space $\, \RO \,$ of rank $1$ torsion-free
sheaves on $\, T^*X \,$, or (since $\, T^*X \,$ is affine if $\, X \,$ is)
of ideal classes in $\, \DO \,$.
If $\, M \,$ is an ideal in $\, \D \,$, we may give
it the filtration induced by that of $\, \D \,$; the associated
graded module $\, \MO := \gr M \,$
is then a (homogeneous) ideal in $\, \DO \,$, and we obtain a well-defined
map $\, \gr : \R \to \RO \,$.  There is also a map
$\, \gammaO : \RO \to \Pic X \,$, which can be defined in several
ways (see Proposition~\ref{4} below).
As the basic one, we choose the following.
Let $\, F \,$ be an ideal in $\, \DO \,$, and let $\, F^* \,$ be the
dual $\, \DO$-module.  Although $\, F \,$ is in general not locally free,
$\, F^* $ is; {\it a fortiori}, the bidual
$\, F^{**} \,$ is locally free, so defines an element of $\, \Pic(T^*X) \,$.
The projection $ \, p : T^*X \to X \,$ induces an isomorphism
$\, \Pic X \cong \Pic(T^*X) \,$: {\it we define $\, \gammaO(F) \,$
to be the class in $\, \Pic X \,$ corresponding to $\, (F^{**}) \,$ under this
isomorphism}.
\begin{theorem}
\la{comm}
The diagram
\begin{equation}
\label{D1}
\begin{diagram}[small]
 \R &                & \rTo^{\gr}  &                 & \RO  \\
    & \rdTo_{\gamma} &             & \ldTo_{\gammaO} &           \\
    &                &     \Pic X  &                 &
\end{diagram}
\end{equation}
commutes.
\end{theorem}

In contrast to our noncommutative case, the
fibration $\, \gammaO \,$ has a natural trivialization: in fact each fibre of
$\, \gammaO \,$ is canonically identified with
the disjoint union of the point
Hilbert schemes $\, \Hilb_n(T^*X) \,$ for $\, n \geq0 \,$:
see Section~\ref{commsec} for more details.

It is a general property of surfaces (we are concerned with
the case $\, T^*X \,$) that the canonical inclusion
$ \, F \hookrightarrow F^{**} \,$ of an ideal of $\, \DO \,$ in its bidual has
finite codimension.  This fact gives us an important invariant
$\, n = n(M) \,$ of an ideal (class) $\, M \,$ in $\, \D \,$: namely,
$\, n(M) \,$ {\it is the codimension of $\, \MO \,$
in its bidual}.  As we shall
see in \cite{BC}, this invariant $\,n \,$ indexes the
Calogero-Moser stratum $\, \C_n \,$ to
which $\, M \,$ belongs in the decomposition of the fibres of $\, \gamma \,$.
In the present paper we content ourselves with proving the following.
\begin{theorem}
\la{picth}
The invariant $\, n \,$ is constant on the orbits of
$\, \Pic \D \,$ in $\, \R \,$.
\end{theorem}

The last theorem we want to formulate in this Introduction makes contact
with the description of $\, \R \,$ given in \cite{CH1}, which (in contrast
to the above) seems to have no classical analogue.  To each ideal class
$\, (M) \in \R \,$ Cannings and Holland assign a {\it primary decomposable}
subspace $\, V \,$ of $\, \O(X) \,$, well defined up to multiplication
by a rational function.  We review the precise definition in
Section~\ref{CHsec};
here we just recall that $\, V \,$ is determined by a finite set of points
$\, x_1 , \ldots, x_k \,$ of $\, X \,$, together with a subspace
$\, V_i \subset \O(X) \,$ associated to each point $\, x_i \,$. Each
$\, V_i \,$ has finite codimension, say $\, n_i \,$, in $\ \O(X) \,$.
Thus we may assign to $\, V \,$ the divisor
\begin{equation}
\la{divdef}
\Div(V) \,:=\,  - \sum_{i=1}^k n_i x_i \
\end{equation}
(the bizarre minus sign is explained by Theorem~\ref{T3} below).
If we multiply $\, V \,$ by a rational function, the class
of $\, \Div(V) \,$ in $\, \Pic X \,$ stays the same: assigning to
$\, (M) \,$ this divisor class, we get another map
\begin{equation}
\label{CH}
\Div : \R \to \Pic X \ .
\end{equation}
\begin{theorem}
\label{T3}
The map \eqref{CH} coincides with $\, \gamma \,$.
\end{theorem}

The paper is arranged as follows.  In Section~\ref{commsec} we discuss
briefly the the commutative analogue of our problem and give the proof
of Theorem~\ref{comm}.  Section~\ref{picsec} is devoted to the action of
$\, \Pic \D $, and contains the proof of Theorems~\ref{T1} and \ref{picth}.
Section~\ref{fatsec} could be read in conjunction with Section~\ref{commsec}:
it contains a more refined version of $\, \gamma \,$ at the level of
ideals (rather than ideal classes).  That is needed in the last section,
which gives the proof of Theorem~\ref{T3}.  Finally, the two appendices
provide proofs for some of the less well known facts used in
Section~\ref{commsec}.

To end this introduction, we point out that our choice to work with
{\it right} ideals in this paper was quite arbitrary, and that there are
similar results for the space $\, \I \,$ of left ideals.
Indeed, the dualization map
$\, M \mapsto \Hom_{\D}(M,\D) \,$ induces a natural bijection
$\, \delta : \R \to \I \,$, and the inclusion of $\, \O(X) \,$
into $\, \D \,$ still
identifies the Grothendieck group of finitely generated {\it left} projective
$\, \D$-modules with $\, \Z \oplus \Pic X \,$.  It follows that we have
a commutative diagram
\begin{equation*}
\begin{diagram}[small]
\R            & \rTo^{\delta}  & \I \\
\dTo^{\gamma} &                & \dTo_{\gamma} \\
\Pic X        & \rTo^{\iota}       & \Pic X \\
\end{diagram}
\end{equation*}
where $\, \iota \,$ is the inverse operation $\, I \mapsto I^* \,$ in the
group $\, \Pic X \,$.  Using this diagram we can deduce results for left
ideals from those for right ideals without rewriting our whole paper.
\subsection*{Notation.}
As above,
if $\, V \,$ is a filtered vector space, the associated graded space
$\, \gr V \,$ will often be denoted by $\, \overline{V}\,$; and if
$\, M \,$ is a module over a ring, we shall write $\, [M] \,$ for the
class of  $\, M \,$ in the appropriate Grothendieck group, and
(if necessary) $\, (M) \,$ for the isomorphism class of $\, M \,$.  To
avoid fussy notation we will sometimes confuse $\, M \,$ and $\, (M) \,$.
We denote by $\, \K \,$ the quotient field of $\, \O(X) \,$ and by
$\, \D(\K) $ the ring of differential operators on $\, \K \,$.

\section{The commutative case}
\la{commsec}
\subsection*{General remarks}
Recall first that if
$\, \F \,$ is a torsion-free coherent sheaf over any smooth variety,
the singular set\footnote{That is, the set of points near
which $\, \F \,$ is not
locally free.} of $\, \F \,$ has codimension at least $2$, and if
$\, \F \,$ is {\it reflexive} its singular set has codimension at least $3$
(see, for example \cite{OSS}, Corollary, p.\ 148 and Theorem~1.1.6, p.\ 145).
In the case of a smooth {\it surface}
$\, S \,$, these simple facts already yield a lot of information about
the space $\, \R(S) \,$ of rank $1$ torsion-free sheaves over $\, S \,$.
Indeed, if $\, \F \,$ is such a sheaf, then its dual
$\, \F^* \,$ (being reflexive) must be locally free, and the canonical
inclusion $\, \F \hookrightarrow \F^{**} \,$ has quotient $\, \F^{**}/\F \,$
supported on a finite set of points, and hence has finite length, say $n$.
Tensoring the inclusion $\, \F \hookrightarrow \F^{**} \,$ with the line
bundle  $\, \F^*\, $, we obtain a sheaf of ideals
$\, \F \otimes \F^* \subseteq \O_S \,$
of colength $n \,$,
that is, a point of the Hilbert scheme $\, \Hilb_n(S) \,$.
In this way we obtain a natural identification
\begin{equation}
\la{RS}
\R(S) \,\cong\, \Pic S  \times  \bigsqcup_{n \geq 0} \Hilb_n(S) \ ;
\end{equation}
the projection
$\, \gamma_S : \R(S) \to \Pic S \,$ onto the first factor in \eqref{RS}
sends the class of $\, \F \,$ to the class of $\, \F^{**} \,$ in
$\, \Pic S \,$.  As we shall see more clearly below, this
is analogous to the map $\, \gamma \,$ that we are studying in the
noncommutative case.

In the case where $\, S \,$ is an {\it affine} surface,
we may replace $\, \F \,$ by the corresponding module $\, F \,$
of global sections, and embed it
as an ideal in $\, \O(S) \,$.  The statements above then translate
into statements about ideals in $\, \O(S) \,$: in particular, an ideal
$\, F \,$ has {\it finite codimension} in its bidual $\, F^{**} \,$.
\subsection*{The map $\gammaO$}
Let us specialize to
the case\footnote{More generally, in the parts of what follows that
are not concerned with $\, \D \,$, $\, S \,$ could be the total
space of any line bundle over $\, X \,$.}
where $\, S = T^*X \,$ for some smooth
affine curve $\, X \,$.
In this case the projection $\, p : T^*X \to X \,$
induces an {\it isomorphism}
$\, \Pic X \cong \Pic(T^*X) \,$ (see Remark~\ref{picis}): we define the map
$\, \gammaO : \RO \to \Pic X \,$ to be the composition of the projection
$\, \gamma_S \,$ above with this isomorphism.
If $\, I \,$
is an ideal in $\, \O(X) \,$, then $\, p^*(I) \in \Pic(T^*X) \,$ is
represented by the extended ideal $\, I\DO \subseteq \DO \equiv \O(T^*X) \,$.
Thus each invertible ideal $\, L \subseteq \DO \,$ is isomorphic to
some $\, I\DO \,$, and $\, I \,$ is determined
up to isomorphism by $\, L \,$.
So the definition of $\, \gammaO \,$ amounts to the following:
{\it if $\, F \,$ is an ideal in $\, \DO \,$, then $\, \gammaO(F) \,$
is the unique ideal class $\, (I) \,$ such that $\, F^{**} \,$ is
isomorphic to $\, I\DO \,$}.

The next Proposition gives several other ways to define $\, \gammaO \,$.
\begin{proposition}
\la{4}
Let $\, F \,$ be an ideal in $\, \DO \,$, and let $\, I \,$ be an ideal
in $\, \O(X) \,$.  Then the following are equivalent.\\
{\rm (i)} $\, F \,$ and $\, I\DO \,$ determine the same class in the
Grothendieck group $\, G_0(T^*X) \,$.\\
{\rm (ii)} $\, F^{**} \,$ and $\, I\DO \,$ determine the same class in the
Grothendieck group $\, K_0(T^*X) \,$.\\
{\rm (iii)} $\, F^{**} \,$ and $\, I\DO \,$ determine the same class in
$\, \Pic(T^*X) \,$.\\
{\rm (iv)} There is an injective homomorphism $\, F \hookrightarrow I\DO \,$
with finite dimensional cokernel.
\end{proposition}
\begin{proof}
(i) $\Rightarrow$ (ii).  Suppose (i) holds.  Then the dual modules $\, F^* \,$
and $\, I^* \DO \,$ also determine the same class in $\, G_0(T^*X) \,$.
Applying this remark again, we find that $\, F^{**} \,$ and $\, I\DO \,$
determine the same class in $\, G_0(T^*X) \,$.  Because $\, T^*X \,$ is
smooth, the canonical map $\, K_0(T^*X) \to G_0(T^*X) \,$ is an
isomorphism, so (ii) follows.\\
(ii) $\Rightarrow$ (iii). This is trivial, since
$\, \det : K_0(T^*X) \to \Pic(T^*X) \,$ maps the class of a line bundle to
itself.\\
(iii) $\Rightarrow$ (iv). We have already
seen that the canonical inclusion $\, F \hookrightarrow F^{**} \,$ has
finite dimensional cokernel, so this follows at once.\\
(iv) $\Rightarrow$ (i). Let $\, T \,$ be the cokernel of the given
inclusion $\, F \hookrightarrow I\DO \,$, so that we have
$\, [F] + [T] = [I\DO] \,$ in $\, G_0(T^*X) \,$.  The support of $\, T \,$
has codimension $2$, so (see Appendix A) $\, [T] \,$ lies in the kernel
of the map
$$
\rk \oplus \det \,:\, G_0(T^*X) \cong K_0(T^*X)
\,\to\, \Z \oplus \Pic(T^*X) \;.
$$
This map is an isomorphism (see Appendix A), hence
$\, [T] = 0 \,$ in $\, G_0(T^*X) \,$ and (i) follows.
\end{proof}

\subsection*{Comparison with the noncommutative case}

Let $\, M \,$ be a rank $1$ torsion-free (right) $\, \D$-module.  We
may embed $\, M \,$ as an ideal in $\, \D \,$ and give it the induced
filtration; we write $\, \MO := \gr M \,$ for the associated graded
$\, \DO$-module.  If we change the embedding of $\, M \,$ in $\, \D \,$,
the induced filtration changes at most by an overall shift, so
the isomorphism class of $\, \MO \,$ (forgetting the filtration) depends
only on the isomorphism class of $\, M \,$.  Thus we get a well-defined
map $\, \gr : \R \to \RO \,$, fitting into the diagram \eqref{D1} in the
Introduction.  We can now give the
\begin{proof}[Proof of Theorem~\ref{comm}]
Let $\, M \,$ be an ideal in $\, \D \,$.  By definition, $\, \gamma(M) \,$
is the unique ideal class $\, (I) \in \Pic X \,$ such that $\, M \,$ is
stably isomorphic to $\, I\D \,$.  That means that for some $\, k \geq 0 \,$,
we have
\begin{equation}
\la{stably}
M \oplus \D^k \, \cong \, I\D \oplus \D^k \ .
\end{equation}
Using the induced filtrations, the graded modules associated to the two
sides of \eqref{stably} are $\, \MO \oplus \DO^k \,$ and
$\, I\DO \oplus \DO^k \,$.  Certainly, these modules are
in general not isomorphic
(because $\, \MO \,$ is not projective); but Ginzburg's lemma (see Appendix B)
implies that they determine the same class in $\, G_0(T^*X) \,$.  By
Proposition~\ref{4} ((i) $\Rightarrow$ (iii)),
this means that $\, \gammaO(\MO) = (I) = \gamma(M) \,$.
\end{proof}
\section{The action of $\Pic \D$}
\la{picsec}

In this section we prove (among other things) Theorems~\ref{T1}
and \ref{picth}.  We shall assume that
$\, X \,$ is not isomorphic to $\, \A^1 \,$: in that case Theorem~\ref{T1}
is vacuous and Theorem~\ref{picth} is already known.  The reason for excluding
$\, \A^1 \,$ is that we shall need to use a theorem of
Makar-Limanov (see \cite{ML}), which
states that except in the case $\, X \cong \A^1 \,$ the ring
$\, \O(X) \subset \D \,$ is characterized as the {\it unique} maximal
abelian ad-nilpotent subalgebra of $\, \D \,$.

We recall from \cite{CH2} some information about $\, \Pic \D \,$.
By definition, the elements of $\, \Pic \D \,$ are the
isomorphism classes of invertible $\, \D$-bimodules
(with scalars $ \c\subset \D $ acting the same way on both sides).
Let $\, P \,$ be an invertible $\, \D$-bimodule. Then Cannings and
Holland show that $\, P \,$ is isomorphic as {\it left} $\, \D$-module to
$\, \D I \,$ for some ideal $\, I \,$ in $\, \O(X) \,$. The (commuting)
structure of right $\, \D$-module is then determined by an
isomorphism $\, \varphi : \D \to \End_{\D}(\D I) \,$:
we denote by $\, (\D I)_{\varphi} \,$
the $\, \D$-bimodule determined in this
way by a pair $\, (I, \varphi) \,$.

Recall that $\, \K \,$ denotes the field of rational functions on $\, X \,$.
In what follows, we shall identify
$\, \End_{\D}(\D I) \,$ with the subring
$\, I^* \D I \,$ of $\, \D(\K) \,$ (where
$\, I^* \subset \K \,$ is the fractional ideal inverse to $\, I \,$):
indeed, we have
\begin{align*}
\End_{\D}(\D I) &\,=\, \{ \theta \in \D(\K) : \D I \theta \subseteq \D I\}\\
                &\,=\, \{ \theta \in \D(\K) : I \theta \subset \D I \}\\
                &\,=\, I^*\D I\ .
\end{align*}
The isomorphism $\, \varphi : \D \to I^* \D I \,$ extends uniquely
to an automorphism of $\, \D(\K) \,$ (which we denote by the same symbol).
By Makar-Limanov's theorem, $\, \varphi \,$
must map $\, \O(X) \,$ to itself, and
hence also must preserve the filtration (by order of operator). We denote by
$\, \sigma \in \Aut(\O(X)) \,$ the restriction\footnote{Generically
$\, \O(X) \,$ has no nontrivial automorphisms, so $\, \sigma \,$
is the identity.}
of $\, \varphi \,$ to $\, \O(X) \,$. Theorem~\ref{T1} follows trivially
from the following explicit formula describing the
(right) action of $\, \Pic \D \,$
on $\, \Pic X \,$.
\begin{proposition}
\la{expl}
Let $\, P := (\D I)_{\varphi} \,$, and let $\, (J) \in \Pic X \,$.
Then we have
$$
J.P \,=\, \sigma^{-1}(JI) \,.
$$
\end{proposition}
\begin{proof}
Note first that
$\, J\D \otimes_{\D} (\D I)_{\varphi} \cong (J\D I)_{\varphi} \,$
(as right $\, \D$-modules); also, $\, \varphi \,$ defines (by restriction)
an isomorphism of right $\, \D$-modules
$$
\varphi \,:\, \sigma^{-1}(JI) \D \,\to\,
(JI)(I^*\D I)_{\varphi} \,=\, (J\D I)_{\varphi} \,.
$$
Since $\, \gamma : \R \to \Pic X \,$ is equivariant with respect to
$\, \Pic \D \,$, we have indeed
\begin{align*}
  J.P &\,=\, \gamma(J\D).P \\
      &\,=\, \gamma(J\D \otimes_{\D} P)\\
      &\,=\, \gamma(\sigma^{-1}(JI) \D)\\
      &\,=\,  \sigma^{-1}(JI) \;,
\end{align*}
as claimed.
\end{proof}
\begin{remark}
Let $\, \Aut \D \,$ be the group of $\, \c$-automorphisms of $\, \D \,$;
then we have the homomorphism
$\, \Aut \D \to \Pic \D \,$ sending $\, \varphi \in \Aut \D \,$ to the
bimodule $\, (\D)_{\varphi} \,$.  The kernel of this map is the group of
{\it inner} automorphisms of $\, \D \,$, so the group $\, \Out \D \,$ of outer
automorphism classes can be regarded as a subgroup of $\, \Pic \D \,$.
These remarks hold for any $\, \c$-algebra $\, \D \,$; in our case
it follows from Proposition~\ref{expl} that the
stabilizer of the neutral element of $\, \Pic X \,$ is exactly
$\, \Out \D \,$,
so we may identify the homogeneous space
$\, \Out \D \,\backslash \Pic \D \,$  with $\, \Pic X \,$.
\end{remark}

Now, the group $\, \Aut \O(X) \,$ acts (on the left) on $\, \O(X) \,$ and
thence on $\, \Pic X \,$, so we can form the semi-direct product
$\, S := \Pic X \rtimes \Aut(\O(X)) \,$.  The multiplication in
$\, S \,$ is given by
$$
(I, \sigma)(J, \tau) \,=\, (I \sigma(J), \sigma \tau) \ .
$$
\begin{lemma}[\cite{CH2}]
\la{hom}
The map\footnote{Cannings and Holland use the map
$\, (\D I)_{\varphi} \mapsto (I^* \!, \sigma) \,$, but this differs from ours
only by an automorphism of $\, S \,$.}
$\, (\D I)_{\varphi} \mapsto (I, \sigma) \,$ induces a well-defined
homomorphism of groups from $\, \Pic \D \,$ onto $\, S \,$.
\end{lemma}
\begin{proof}[Sketch of proof.]
Recall that $\, \sigma := \varphi \,\vert\, \O(X) \,$.  The lemma
follows from the fact that the multiplication of the bimodules
$\, (\D I)_{\varphi} \,$ is given by
$$
(\D I)_{\varphi} \otimes_{\D} (\D J)_{\psi} \,\cong\,
(\D I \sigma(J))_{\varphi \psi} \ .
$$
The isomorphism is defined by
$\, D_1 \otimes_{\D} D_2 \mapsto D_1 \varphi(D_2) \,$.
See \cite{CH2} for more details.
\end{proof}

The semi-direct product $\, S \,$ acts naturally (on the right) on
$\, \DO$-modules $\, F \,$ by
\begin{equation}
\la{S}
F.(I, \sigma) \,= \, \widehat{\sigma}_* (F \otimes_{\DO} I\DO) \ ,
\end{equation}
where $\, \widehat{\sigma} \,$ denotes the natural lifting of $\, \sigma \,$
to an automorphism of $\, \DO \equiv \O(T^*X) \,$.
In the case where $\, F \,$ is an ideal in $\, \DO \,$ that is equivalent to
\begin{equation}
\la{S1}
F.(I, \sigma) \,= \, \widehat{\sigma}^{-1} (FI) \ .
\end{equation}
Thus $\, S \,$ acts naturally on the space of ideals
$\, \RO \,$: using Lemma~\ref{hom}, we
may regard this as an action of $\, \Pic \D \,$.
Theorem~\ref{picth} will follow easily from the next
\begin{proposition}
\la{greq}
The map $\, \gr : \R \to \RO \, $ is equivariant with respect to
the actions of $\, \Pic \D \,$.
\end{proposition}

To see this, we need one more lemma. The graded algebra associated to
$\, I^* \D I \,$ is canonically identified with $\, \DO \,$, so that
$\, \varphi \,$ induces a graded automorphism $\, \overline{\varphi} \,$
of $\, \DO \,$.
\begin{lemma}
\la{phi}
We have $\, \overline{\varphi} = \widehat{\sigma} \,$.
\end{lemma}
\begin{proof}
Recall that $\, \DO \,$ is just the symmetric algebra of the module
$\, \Der \O(X) \,$ of derivations of $\, \O(X) \,$. By construction,
$\, \overline{\varphi} \,$ and $\, \widehat{\sigma} \,$ coincide on
$\, \O(X) \,$, so we have only to see that they agree on
$\, \Der \O(X) \,$.  Now, the identification
$$
I^* \D_{\leq 1} I / \O(X) \,\to\, \DO_1 \equiv \Der \O(X)
$$
sends an operator $\, D \,$ to the derivation $\, f \mapsto [D,f] \,$.
Also, if $\, \partial \in \Der \O(X) \,$ \and
$\, f \in \O(X) \,$, then in $\, \D \,$
we have the relation $\, [\partial, f] = \partial(f) \,$, whence
\begin{align*}
[\varphi(\partial),\,f] &\,=\, [\varphi(\partial), \,\varphi(\sigma^{-1}(f))]\\
                        &\,=\, \varphi[\partial, \,\sigma^{-1}(f)]\\
                        &\,=\, \varphi(\partial(\sigma^{-1}(f)))\\
                        &\,=\, (\sigma \partial \sigma^{-1})(f) \;.
\end{align*}
Combining these two remarks, we find that
$\, \overline{\varphi}(\partial) = \sigma \partial \sigma^{-1} \,$, which
(by definition) is indeed $\, \widehat{\sigma}(\partial) \,$.
\end{proof}

Now we can give the
\begin{proof}[Proof of Proposition~\ref{greq}]
Let $\, M \,$
be an ideal in $\, \D \,$, and let $\, P = (\D I)_{\varphi} \,$ be
an invertible $\, \D$-bimodule, as above.
Then $\, M \otimes_{\D} P = MI  \,$
with the structure of (right) $\, \D$-module
determined via the isomorphism $\, \varphi \,$.
Since $\, \varphi \,$ is filtration-preserving, it follows that
$\, \overline{M \otimes_{\D} P} = \MO I \,$ with the structure of
$\, \DO$-module determined via $\, \overline{\varphi} \,$; that is,
using Lemma~\ref{phi},
$$
\gr(M \otimes_{\D} P) \,\equiv\,
\overline{M \otimes_{\D} P} \,\cong\, \widehat{\sigma}^{-1}(\MO I) \ .
$$
Comparing with \eqref{S1}, we see that the map
$\, \gr  \,$ is equivariant.
\end{proof}

\begin{proof}[Proof of Theorem~\ref{picth}]
It is clear that if $\, F \,$ is an ideal in $\, \DO \,$ then the action
\eqref{S1}
of $\, S \,$ preserves the codimension of $\, F \,$ in $\, F^{**} $, so
Theorem~\ref{picth} follows at once from Proposition~\ref{greq}.
\end{proof}

\begin{remark*}
Comparing Proposition~\ref{expl} with the formula \eqref{S1} shows that
our map $\, \gammaO : \RO \to \Pic X \,$ is equivariant with respect to
the actions of $\, S \,$ (which we may regard as actions of $\, \Pic \D \,$).
Thus {\it all} the maps in the commutative diagram \eqref{D1}
are equivariant with respect to the actions of $\, \Pic \D \,$.
\end{remark*}

\section{Fat ideals}
\la{fatsec}

Following [BGK], we call an ideal in $\, \D \,$ or $\, \DO \,$
{\it fat} if its intersection with $\, \O(X) \,$ is nonzero.  The usefulness
of this notion comes from the simple
\begin{lemma}
\la{fat}
Every ideal in $\, \D \,$ or $\, \DO \,$ is isomorphic to a fat one.
\end{lemma}
\begin{proof}
We give the proof for $\, \DO \,$: the proof for $\, \D \,$ is the same
(see \cite{St}, Lemma 4.2).
Recall again that $\, \K \,$ is the quotient field of $\, \O(X) \,$.
The algebra
$\, \K \otimes_{\O(X)} \DO \,$ is a principal ideal domain
(because it is a polynomial algebra $\, \K[\xi] \,$:
for $\, \xi \,$ we can take any
nonzero derivation of $\, \K \,$).  Given an ideal $\, F \subseteq \DO \,$,
let $\, a \,$ be a generator of its extension to
$\, \K \otimes_{\O(X)} \DO \,$.
Then the fractional ideal $\, a^{-1}F \,$ has nonzero intersection
with $\, \K \,$; multiplying by a suitable element of $\, \O(X) \,$ we
get a fat (integral) ideal isomorphic to $\, F \,$.
\end{proof}
\begin{lemma}
\la {fati}
Let $\, L \,$ be an invertible fat ideal in  $\, \DO \,$, and let
$\, I := L \cap \O(X) \,$.  Then $\, L = I \DO \,$.
\end{lemma}
\begin{proof}
Clearly, $\, I \DO \subseteq L \,$.  The divisor of zeros of $\, I \DO \,$
is a linear combination of fibres of the projection
$\, p : T^*X \to X \,$; the divisor of zeros of $\, L \,$ must be
smaller, that is, $\, L = J \DO \,$ for some $\, J \supseteq I \,$.
But then $\, L \cap \O(X) = J \,$, so $\, J = I \,$.
\end{proof}
\begin{remark}
\la{picis}
We have essentially given a proof of the fact (which we used earlier)
that the map $\, p^* : \Pic X \to \Pic(T^*X) \,$ is an isomorphism. Indeed,
$\, p^* \,$ is obviously injective,
because if $\, i : X \to E \,$ is the inclusion of the zero section,
then $\, pi = \id_X \,$, so $\, i^* p^* \,$ is the identity on
$\, \Pic X \,$.  Lemmas \ref{fat} and \ref{fati} show
that $\, p^* \,$ is surjective.
\end{remark}
\begin{proposition}
\la{fatI}
Let $\, F \,$ be a fat ideal in $\, \DO \,$. Then there is a unique
ideal $\, I \,$ in $\, \O(X) \,$ such that $\, F \,$ has finite
codimension in $\, I \DO \,$.  Furthermore, $\, \gammaO (F) = (I) \,$.
\end{proposition}
\begin{proof}
We may realize the dual module $\, F^* \,$ of $\, F \,$ as the
fractional ideal
$$
F^* \,=\, \{ f \in \Frac(\DO) : fF \subseteq \DO \} \;.
$$
With this understanding, the map $\, F \mapsto F^* \,$ reverses inclusions,
and the canonical inclusion $\, F \hookrightarrow F^{**} \,$ becomes
an inclusion of ideals in $\, \DO \,$. If $\, F \,$ is fat, so is
$\, F^{**} \,$, so Lemma~\ref{fati} gives $\, F^{**} = I\DO \,$, where
$\, I := F^{**} \cap \O(X) \,$. This shows existence of the ideal $\, I \,$ in
the proposition.  To see uniqueness,
suppose that also $\, F \subseteq J\DO \,$
with finite codimension.  Taking biduals, we get
$\, F \subseteq F^{**} \!= I\DO \subseteq J\DO \,$.
Thus $\, I \subseteq J \,$
and $\, J\DO / I\DO \,$ has finite dimension, which is impossible
unless $\, I = J \,$.  Finally, the last assertion in the Proposition
comes from the fact that $\, \gammaO \,$ can be defined by the property
(iv) in Proposition~\ref{4}.
\end{proof}

Clearly, if $\, M \,$ is fat, so is $\, \MO \,$.  So
Proposition~\ref{fatI} implies:
\begin{theorem}
\la{fatM}
Let $\, M \subseteq \D \,$ be a fat ideal.  Then there is a unique ideal
$\, I \,$ in $\, \O(X) \,$ such that $\, \MO \,$ has finite codimension
in $\, I\DO \,$.  Furthermore, $\, I \DO = \MO^{**} \!$, and
$\, \gamma(M) = (I) \,$ in $\, \Pic X \,$.
\end{theorem}
\section{The Cannings-Holland correspondence}
\la{CHsec}
Recall from \cite{CH1} that a subspace $\, V \subseteq \O(X) \,$
is called {\it primary} (more precisely, {\it $x$-primary})
if it contains a power of the maximal ideal $ \m_x $ of functions that
vanish at a point $ x \in X $
(we write $ V = V_x $ in this case).
Now, $\, V \subseteq \O(X) $ is called {\it  primary decomposable}
if it is an intersection of primary subspaces $ V_x $ with $ V_x = \O(X) $
for almost all $ x \in X $.
By \cite{CH1}, Theorem 2.4, the primary decomposition of $ V $ is uniquely
determined. If $\, V \,$ is a primary decomposable subspace
of $\, \O(X) \,$, we set
$$
M_V \,:=\, \{D \in  \D \,:\, D.\O(X) \subseteq V \}
$$
(here $\, D.f \,$ denotes the function obtained by letting $\, D \,$
act on $\, f \,$, not to be confused with the operator
$\, Df \in \D \,$).  Clearly, $\, M_V \,$ is a right ideal in
$\, \D \,$.  The main result of \cite{CH1} was the following.
\begin{theorem}
\la{CHth}
The map $\, V \mapsto M_V \,$ is a $1$-$1$ correspondence between the
lattice of primary decomposable subspaces of $\, \O(X) \,$ and the
lattice of fat right ideals in $\, \D \,$.  The inverse map is
given by
$\, M \mapsto M.\O(X) \,$.
\end{theorem}

Theorem~\ref{fatM} associates to each fat ideal
$\, M \subseteq \D \,$ an ideal $\, I \,$ in $\, \O(X) \,$ (namely,
the ideal such that $\, \MO^{**} \!= I \DO \,$).  If $\, V \,$ is a
primary decomposable subspace of $\, \O(X) \,$ we write $\, I_V \,$
for the ideal associated in this way to $\, M_V \,$. The relationship
of this construction to Theorem~\ref{CHth} is as follows.
\begin{theorem}
\la{hard}
Let $\, V \,$ be the intersection of $x$-primary subspaces $\, V_x \,$,
and let $\, n_x \,$ be the codimension of $\, V_x \,$ in $\, \O(X) \,$.
Then we have
\begin{equation}
\la{div}
I_V \,= \, \prod_{x \in X} \m_{x}^{n_x} \ .
\end{equation}
\end{theorem}

The proof rests on the following lemma.
\begin{lemma}
\la{vhard}
Let $\, V \subseteq \O(X) \,$ be primary decomposable,
$\, I_V \subseteq \O(X) \,$ the corresponding ideal.  Then
\begin{equation}
\la{dims}
\dim(\O/V) \,=\,  \dim(\O/I_V) \ .
\end{equation}
\end{lemma}
\begin{proof}
We start with the exact sequence of $\, \DO$-modules
\begin{equation}
\la{exact}
0 \,\to\, I_V \DO / \MO_V \,\to\, \DO / \MO_V
\,\to\, \DO / I_V \DO \,\to\, 0 \; .
\end{equation}
Recall from \cite{SS} that the $1$-{\it length} of a $\, \DO$-module is
the maximum length of a chain of submodules such that the successive
quotients are infinite-dimensional (over $\, \c \,$).  Since the $1$-length
is additive, and the first term in \eqref{exact} has finite dimension
(and hence $1$-length $0$), we get
$$
\ol(\DO / \MO_V) \,=\, \ol(\DO / I_V \DO) \ .
$$
By \cite{SS}, Lemma~3.10, this is equivalent to
$$
\len(\D / M_V) \,=\, \len(\D / I_V \D) \ .
$$
By \cite{CH1}, Corollary~3.8, this in turn is equivalent to \eqref{dims}.
\end{proof}

Now we can give the proof of Theorem~\ref{hard}. To simplify the
notation, if $\, V_x \,$ is $x$-primary, we shall write simply
$\, M_x \,$ and $\, I_x \,$ for the associated ideals in $\, \D \,$
and in $\, \O(X) \,$.
\begin{proof}[Proof of Theorem~\ref{hard}]
We consider first the case where $\, V \equiv V_x \,$ is $x$-primary:
thus $\, \m_{x}^N \subseteq V_x \,$ for some $\, N \,$;
it follows that $\, \m_{x}^N \D \subseteq  M_x \,$.
Passing to the associated graded modules, we get
$$
\m_{x}^N \DO \subseteq \MO_x \subseteq \MO_x^{**} = I_x\DO \, ,
$$
and hence  $\, \m_{x}^N \subseteq I_x \,$.
By the uniqueness of factorization
of ideals in the Dedekind domain $\, \O(X) \,$, this means that
$\, I_x = \m_{x}^k \,$ for some $\, k \leq N \,$.  Lemma~\ref{vhard}
then gives that $\, n_x = k  \,$; that is, we have proved Theorem~\ref{hard}
in the case where $\, V \,$ is $x$-primary.

In general, suppose that $\, V \,$ is the intersection of primary
subspaces $\, V_x \,$; then we have
$\, V \subseteq V_x \,$ for each $\,x \,$.
It follows that $\, M_V \subseteq M_x \,$, and hence
$\, \MO_V \subseteq \MO_x \,$.  Taking biduals (inside the
field of fractions of $\, \DO \,$, as in the proof of Theorem~\ref{fatI})
we get
$\, I_V \DO \subseteq I_x \DO \,$, and hence
$\, I_V \subseteq I_x = \m_{x}^{n_x} \,$.  This means that if the
factorization of $\, I_V \,$ is given by $\, I_V = \prod \m_{x}^{k_x} \,$,
then $\, n_x \leq k_x \,$ for all $\, x \,$.  But by \cite{CH1}, Theorem 2.4,
the left hand side of \eqref{dims} is equal to $\, \sum n_x \,$, so
Lemma~\ref{vhard} says that $\, \sum n_x = \sum k_x \,$.  It follows
that $\, n_x = k_x \,$ for all $\, x \,$.
\end{proof}

Theorem~\ref{hard} is in essence a more precise version of
Theorem~\ref{T3} from the Introduction.  To see that, note first that
the fat representative of an ideal class in $\, \R \,$ is unique up to
multiplication by a rational function; that is, $\, \R \,$ is the
quotient of the space of fat ideals by the equivalence
relation\footnote{Our exposition would be smoother at this point if we
worked with {\it fractional} ideals in $\, \D(\K) \,$.}
$$
M_1 \sim M_2 \, \Longleftrightarrow f M_1 = g M_2 \text{\, for some \,}
f,\,g \in \O(X) \;.
$$
Obviously, if $\, f \in \O(X) \,$, then Theorem~\ref{CHth} makes
$\, fV \,$ correspond to $\, f M_V \,$; also, if $\, M \,$ is a fat
ideal, then $\, \overline{fM} = f \MO \,$; hence for any primary
decomposable $\, V \,$ we have $\, I_{fV} = f I_V \,$.  In the usual
correspondence between ideals and divisors, the ideal $\, I_V \,$ in
\eqref{div} corresponds to the divisor $\, \Div(V) \,$ in \eqref{divdef}.
It follows that we have
$$
\Div(fV) = \Div(V) - \Div(f)
$$
(where $\, \Div(f) \,$ is the divisor of zeros of $\, f \,$).  Thus
(as claimed in the Introduction) the map $\, V \mapsto \Div(V) \,$ yields
a well-defined map $\, \Div : \R \to \Pic X \,$. That this map coincides
with $\, \gamma \,$ follows from Theorem~\ref{fatM}.

\appendix
\section{$SK_0$ of algebraic varieties}

Here we provide proofs of a couple of (well known) facts for which
we do not know a convenient reference. We need these facts only for
affine varieties $\, X \,$; however, they are valid for any smooth
quasi-projective variety, indeed, for any regular Noetherian scheme.
We denote by $\, K_0(X) \,$ the
Grothendieck group of vector bundles over
$\, X \,$; in the affine case this is the same
as the Grothendieck group of finitely generated projective $\, \O(X)$-modules.
Taking the determinant of a vector bundle induces a map
$\, \det : K_0(X) \to \Pic X \,$.  Considering also the rank of a vector
bundle, we obtain a map (obviously surjective)
\begin{equation}
\la{rkdet}
\rk \oplus \det \,:\, K_0(X) \to \Z \oplus \Pic X \ .
\end{equation}
The kernel of this map is denoted by $\, SK_0(X) \,$.
If $\, X \,$ is a curve, then $\, SK_0(X) = 0 \,$
(see, for example, \cite{H}, Ch.\ II, Ex.\ 6.11).
In the proof of Proposition~\ref{4},
we used also the fact that if $\, X \,$ is a curve, then
$\, SK_0(T^*X) = 0 \,$. That is a special case of the following.
\begin{proposition}
Let $\, E \,$ be the total space of a vector bundle
over a smooth variety $\, X \,$.
Then $\, SK_0(E) = 0 \,$ if and only if $\, SK_0(X) = 0 \,$.
\end{proposition}
\begin{proof}
Let $\, p : E \to X \,$ be the projection.  Then we have a commutative
diagram
\begin{equation}
\la{D}
\begin{diagram}[small]
K_0(X)      &\rTo^{\rk \oplus \det} &\Z \oplus \Pic X \\
\dTo^{p^*}  &                       &\dTo_{\id \oplus p^*}\\
K_0(E)      &\rTo^{\rk \oplus \det} &\Z \oplus \Pic E\\
\end{diagram}
\end{equation}
in which the two horizontal arrows are surjective.
By \cite{B}, Theorem 3.2, p.\ 636 in the affine case, or
\cite{Q}, Proposition 4.1, p.\ 128 in general,
the left hand vertical arrow $\, p^* \,$ is an isomorphism.
It follows from the commutativity of the diagram \eqref{D} that the
right hand vertical arrow is surjective, hence it too is an isomorphism
(since it is obviously injective, cf.\  Remark~\ref{picis}).
The Proposition is now obvious.
\end{proof}

As well as $\, K_0(X) \,$ we have the Grothendieck group $\, G_0(X) \,$
of coherent sheaves over $\, X \,$; in the affine case this is the same
as the Grothendieck group of finitely generated $\, \O(X)$-modules.
For {\it smooth} $\, X \,$, the natural map
$\, K_0(X) \to G_0(X) \,$ is an isomorphism (see, for example, \cite{Q},
p.~124), so we may identify these two groups.  We then have
another important fact about $\, SK_0(X) \,$: it is generated by the
classes (in $\, G_0(X) \cong K_0(X) \,$)
of sheaves whose support has codimension $\, \geq 2 \,$.  Here
is a proof of the part of that which we used.
\begin{proposition}
Let $\, \F \,$ be a coherent sheaf over $\, X \,$ whose support
has codimension  $\, \geq 2 \,$.  Then
$\, [\F] \in SK_0(X) \,$.
\end{proposition}
\begin{proof}
Obviously, the rank of $\, \F \,$ is zero, so we have only to consider
its determinant.  Choose a resolution
$$
0 \to P_n \to \ldots \to P_1 \to P_0 \to \F \to 0
$$
of $\, \F \,$ by vector bundles $\, P_i \,$.  By definition, $\, \det \F \,$
is the alternating product of the determinant line bundles of the
$\, P_i \,$.  This line bundle is trivial off the support
of $\, \F $, that is, it has a rational section whose divisor of zeros and
poles is contained in the support of $\, \F \,$.
Since the support has codimension
greater than $1$, that is impossible unless the divisor of zeros and
poles is empty, that is, $\, \det \F \,$ is trivial, as claimed.
\end{proof}
\section{A lemma of Ginzburg on filtrations}
Let $ A $ be a positively filtered associative ring such that
$ \AO := \gr{A} $ is right Noetherian.
Let $ G_0(\AO) $ be the Grothendieck group of
finitely generated right $ \AO$-modules, and let
$ M $ be a finitely generated right $A$-module equipped with a good
filtration (recall that {\it good} means that the associated graded
$\, \AO$-module $\, \MO \,$ is finitely generated).  As usual, we write
$ [\,\MO\,] $ for the class of $ \MO $ in $ G_0(\AO) $.
The following observation is (apparently) due to Ginzburg (see \cite{G},
Corollary~1.3, p.~338).
\begin{lemma}
\label{L1}
$ [\,\MO \,] $ does not depend on the choice of good filtration on $M$.
\end{lemma}
\begin{proof}
Call two filtrations $ M_{\bullet} $ and $ M_{\bullet}' $
on $ M $ {\it adjacent} if
$\, M_{k-1} \subseteq M_k' \subseteq M_k \,$ for all $ k \in \Z $. Given
a pair of adjacent
filtrations, set $ \overline{L} := \bigoplus_{k \in \Z} M_k'/M_{k-1} $ and
$ \overline{N} := \bigoplus_{k \in \Z} M_k/M_k' $. Then we have the obvious
exact sequences of (graded) $ \AO$-modules:
$$
0 \to \overline{L} \to \MO \to \overline{N} \to 0\qquad \mbox{and}
\qquad
0 \to \overline{N}(-1) \to \MO' \to \overline{L} \to 0 \ .
$$
Forgetting the grading, we have $ \overline{N}(-1) \cong \overline{N} $,
and therefore
\begin{equation}
\la{E0}
[\,\MO\,] = [\,\overline{L}\,] + [\,\overline{N}\,] =
[\,\overline{L}\,] + [\,\overline{N}(-1)\,] =
[\,\MO'\,]  \quad \mbox{in}\  G_{0}(\AO)\ .
\end{equation}
Now, for any filtrations $ M_{\bullet} $ and $ M_{\bullet}' $ on $M$,
define a sequence of adjacent filtrations
$ \{M^{(j)}_{\bullet}\,:\, j \in \Z \} $ by $\, M^{(j)}_{k} :=
M_k + M_{k-j}' \,$
so that $\, M^{(j)}_{k-1} \subseteq M^{(j+1)}_{k}
\subseteq M^{(j)}_{k} \,$ for all $ k, j \in \Z$.
If $ M_{\bullet} $ and $ M_{\bullet}' $ are both good, there is an
$\, n \in \N \,$ such that $\, M_{k-n}'
\subseteq M_{k} \subseteq M_{k+n}'\,$ for all
$ k \in \Z $ (see, for example, \cite{Bj}, Proposition~1.15, p.~458).
Thus, in the case of good filtrations we have
$\, M^{(j)}_{\bullet} = M_{\bullet}' \,$
for $\, j \leq -n \,$ and $\, M^{(j)}_{\bullet} = M_{\bullet} \,$ for
$\, j \geq n \,$,
so the lemma follows from \eqref{E0}.
\end{proof}

\bibliographystyle{amsalpha}

\end{document}